\documentclass[11pt,a4paper,reqno]{amsart}

\usepackage{amssymb, amsmath, textcomp, amsthm}
\usepackage{bbm,dsfont}
\numberwithin{equation}{section}

\title{On a binary system of Prendiville: The cubic case}
\author{Shaoming Guo}
\date{}

\def\R{\mathbb{R}}
\def\N{\mathbb{N}}
\def\C{\mathbb{C}}\def\nint{\mathop{\diagup\kern-13.0pt\int}}

\def\lesim{\lesssim}

\def\beq{\begin{equation}}
\def\endeq{\end{equation}}
\def\bg{\begin{gathered}}
\def\eg{\end{gathered}}

\def\83{\frac{8}{3}}
\def\38{\frac{3}{8}}

\def\mc{\mathcal}

\theoremstyle{plain}
\newtheorem{thm}{Theorem}[section]
\newtheorem{prop}[thm]{Proposition}

\newtheorem{lem}[thm]{Lemma}

\newtheorem*{conj*}{Conjecture}

\newtheorem*{openproblem*}{Open Problem}

\begin{document}

\maketitle

\begin{abstract}
We prove sharp decoupling inequalities for a class of two dimensional non-degenerate surfaces in $\R^5$, introduced by Prendiville \cite{Pre}. As a consequence, we obtain sharp bounds on the number of integer solutions of the Diophantine systems associated with these surfaces. 
\end{abstract}

\smallskip

\let\thefootnote\relax

\section{Introduction}

Let $\Phi(t, s)$ be a homogeneous polynomial of degree three. Consider the two dimensional surface
\beq\label{surface}
\mc{S}=\{(t, s, \Phi_{t}(t, s), \Phi_{s}(t, s), \Phi(t, s)): (t, s)\in [0, 1]^2\}.
\endeq 
We say that $\Phi$ is non-degenerate if it can not be written as $(\mu t+\nu s)^3$ for any $\mu, \nu\in \R$. This is the same as saying that, if we write $\Phi(t, s)=at^3+ bt^2 s+c ts^2+d s^3$, then the matrix
\beq\label{non-dege-matrix}
\begin{pmatrix}
3a & 2b & c\\
b & 2c & 3d
\end{pmatrix}
\endeq
has rank two. 
This will be our assumption throughout the present paper. \\

Consider the following system of Diophantine equations
\beq\label{diophantine}
\begin{cases}
\hfill x_{1}+x_{2}+...+x_{r} \hfill & =x_{r+1}+x_{r+2}+...+x_{2r},\\
\hfill y_{1}+y_{2}+...+y_{r}\hfill & =y_{r+1}+y_{r+2}+...+y_{2r},\\
\hfill \Phi_t(x_1, y_1)+...+\Phi_t(x_r, y_r)\hfill & =\Phi_t(x_{r+1}, y_{r+1})+...+\Phi_t(x_{2r}, y_{2r}),\\
\hfill \Phi_s(x_1, y_1)+...+\Phi_s(x_r, y_r)\hfill & =\Phi_s(x_{r+1}, y_{r+1})+...+\Phi_s(x_{2r}, y_{2r}),\\
\hfill \Phi(x_1, y_1)+...+\Phi(x_r, y_r)\hfill & =\Phi(x_{r+1}, y_{r+1})+...+\Phi(x_{2r}, y_{2r}).
\end{cases}
\endeq
Here $r$ is a positive integer and $x_i, y_i\in \N$ for each $1\le i\le 2r$. For a large integer $N$, we let $J_r(N)$ denote the number of integer solutions $(x_1, ..., x_{2r}, y_1, ..., y_{2r})$ of the system \eqref{diophantine} with $0\le x_i, y_i\le N$ for each $1\le i\le 2r$. We prove 
\begin{thm}\label{main-1}
For each $r\ge 1$ and each $\epsilon>0$, we have 
\beq\label{main-1-dio}
J_r(N) \lesim_{r, \epsilon} N^{2r+\epsilon}+N^{4r-9+\epsilon}.
\endeq
Here the implicit constant depends only on $r$ and $\epsilon$. Moreover, up to the arbitrarily small factor $\epsilon$, the exponents of $N$ are sharp. 
\end{thm}

The lower bounds have been calculated by Parsell, Prendiville and Wooley \cite{Par}. Our focus is to obtain the upper bounds \eqref{main-1-dio}. This will be done via proving a sharp decoupling inequality.

For a measurable set $R\subset [0, 1]^2$ and a measurable function $g: R\to \C$, define the extension operator associated with $\mc{S}$ by 
\beq
E_R g(x)=\int_R g(t, s)e^{itx_1+isx_2+i\Phi_t(t, s)x_3+i\Phi_s(t, s)x_4+i\Phi(t, s)x_5} dtds. 
\endeq 
Here $x=(x_1, \dots, x_5)$. For a ball $B=B(c, R)\subset \R^5$ with center $c$ and radius $R$, we use the weight 
\beq
w_B(x)=(1+\frac{\|x-c\|}{R})^{-C},
\endeq
where $C$ is a large enough constant whose value will not be specified. For each $2\le q\le p$ and $0<\delta<1$, let $B_{p, q}(\delta)$ be the smallest constant such that 
\beq\label{decoupling-result}
\|E_{[0, 1]^2}g\|_{L^p(w_{B})}\le B_{p,q}(\delta) (\sum_{\substack{\Delta: \text{ square in } [0, 1]^2\\ l(\Delta)=\delta}}\|E_{\Delta}g\|_{L^p(w_{B})}^q)^{1/q},
\endeq
holds for each ball $B\subset \R^5$ of radius $\delta^{-3}$. Inequalities of this type are referred to as $l^q L^p$ decouplings. Via a standard reduction (see for instance Section 2 \cite{BD2}), Theorem \ref{main-1} follows from 
\begin{thm}\label{main-theorem-2}
We have 
\beq\label{desired-decoupling}
B_{9, 9}(\delta)\lesim (\frac{1}{\delta})^{2(\frac{1}{2}-\frac{1}{9})+\epsilon},
\endeq
for each $\epsilon>0$ and $0<\delta\le 1$.
\end{thm}

The system \eqref{diophantine} is the cubic case of a system considered by Prendiville \cite{Pre}. The way that the surface \eqref{surface} and the Diophantine system \eqref{diophantine} are formulated is slightly different from those in Prendiville \cite{Pre}. There the surface \eqref{surface} is replaced by 
\beq
\mc{S}'=\{(\Phi_{tt}(t, s), \Phi_{ts}(t, s), \Phi_{ss}(t, s), \Phi_{t}(t, s), \Phi_{s}(t, s), \Phi(t, s)): (t, s)\in [0, 1]^2\}.
\endeq
That is, the surface $\mc{S}'$ is obtained by taking successive partial derivatives of the seed polynomial $\Phi$. However, under the non-degeneracy condition that the matrix \eqref{non-dege-matrix} has rank two, we observe that the vector space $[\Phi_{tt}, \Phi_{ts}, \Phi_{ss}]$ is always the same as $[t, s]$.  Hence the system of Diophantine equations associated with the surface $\mc{S}'$ is always equivalent to that associated with the surface $\mc{S}$, in the sense that they admit the same number of integer solutions. 

To obtain a system analogous to \eqref{diophantine} of higher degrees, one takes a seed polynomial $\Phi(t, s)$ of degree $k\ge 3$, extracts all the partial derivatives 
\beq\label{binary-2}
\frac{\partial^{i_1+i_2}\Phi(t, s)}{\partial t^{i_1} \partial s^{i_2}} \ \ \  ( i_1\ge 0, i_2\ge 0),
\endeq
and forms a Diophantine system by using all these partial derivatives. If we take $\Phi(t, s)$ to be the monomial $t^{k_1} s^{k_2}$ with $k_1\ge k_2\ge 1$, then we recover the so-called simple binary systems 
\beq\label{binary-3}
x^{i_1}_1y^{i_2}_1+\dots +x^{i_1}_ry^{i_2}_r=x^{i_1}_{r+1}y^{i_2}_{r+1}+\dots +x^{i_1}_{2r}y^{i_2}_{2r}, \text{ with } i_1\le k_1, i_2\le k_2 \text{ and } (i_1, i_2)\neq (0, 0),
\endeq
which appeared in recent work in quantitative arithmetic geometry (Section 4.15 \cite{tsc09} and \cite{val11}). Notice that if we take $\Phi$ to be a polynomial of degree $k$ that depends only on one variable, then we recover the Vinogradov system 
\beq
x^i_1+\dots +x^i_r=x^i_{r+1}+\dots +x^i_{2r}, \text{ with } 1\le i\le k.
\endeq
All the systems mentioned above fall into the framework of translation-dilation invariant systems, which are intensively studied in \cite{Par}. In our setting, this is reflected in the validity of the parabolic rescaling lemma (Lemma \ref{abc18}).

Parsell, Prendiville and Wooley \cite{Par} proved \eqref{main-1-dio} for $r\ge 21$, using the method of efficient congruencing. In the current paper we prove it for all $r\ge 1$, using the decoupling theory developed in \cite{BD1} and \cite{BDG}. When intending to generalise our proof to the above binary systems (\eqref{binary-2} or \eqref{binary-3}) of degrees higher than three, one encounters enormous difficulties. In comparison, the efficient congruencing method still provides bounds that are almost optimal. We refer to \cite{Par} for the precise statement of the corresponding results. \\

Let us mention a further application of the result in Theorem \ref{main-theorem-2}. This application has been worked out carefully in \cite{Pre}, \cite{Par} and \cite{Hen}, hence we mention it briefly. Let $\Phi$ be as above, a homogeneous polynomial of degree three that is non-degenerate. Take $r\in N$. Let $c_1, c_2, \dots, c_r$ with $c_1+c_2+\dots +c_r=0$ be a ``non-singular'' (Definition 1.1 \cite{Pre}) choice of coefficients for $\Phi$. Consider the equation 
\beq\label{roth}
c_1\Phi(x_1, y_1)+c_2\Phi(x_2, y_2)+\dots +c_r\Phi(x_r, y_r)=0.
\endeq
The solution $\{(x_1, y_1), \dots, (x_r, y_r)\}$ to the above equation is called diagonal is they all lie on a line in the plane. Take a large number $N\in \N$. Let $A\subset [0, N]^2$ be a set which contains only diagonal solutions to the equation \eqref{roth}. Then a result in \cite{Pre} (further improved in \cite{Par}) states that 
\beq
|A| \ll N^2 (\log \log N)^{-1/(s-1)},
\endeq
for $s$ bigger than certain threshold. The validity of the estimate \eqref{desired-decoupling} will further lower down this threshold. We refer the interested reader to \cite{Pre} and \cite{Hen} for the details. \\

In the end, we mention some novelties of our proof and explain briefly the potential difficulties that appear when trying to adapt our argument to binary systems of higher degrees. 

In decoupling theory, various Brascamp-Lieb inequalities (see \eqref{0510e2.2}) play fundamental roles. In order to apply these inequalities, one needs to check a transversality condition (see \eqref{bl-transversality}). When the dimensions and co-dimensions of the surfaces under consideration get higher and higher, checking these transversality conditions will become more and more difficult. In the current paper, we are dealing with a two dimensional surface in $\R^5$. To check \eqref{bl-transversality}, we further develop the idea introduced in \cite{BDG-2}, where a specific two dimensional surface in $\R^9$ is considered. As currently we are dealing with a class of surfaces, certain algebraic structures need to be explored. For instance, see Subsection \ref{subsection3.2}, in particular Lemma \ref{general-taylor}. \\

{\bf Notation:} Throughout the paper we will write $A\lesssim_{\upsilon}B$ to denote the fact that $A\le CB$ for a certain implicit constant $C$ that depends on the parameter $\upsilon$. Typically, this parameter is either $\epsilon$ or $K$. The implicit constant will never depend on the scale $\delta$, on the balls we integrate over, or on the function $g$. It will however most of the times depend on the Lebesgue index $p$.

We will denote by $B_R$ an arbitrary ball of radius $R$. We use the following notation for averaged integrals
$$\|F\|_{L^p_\sharp(w_B)}=(\frac1{|B|}\int|F|^pw_B)^{1/p}.$$
$|A|$ will refer to either the cardinality of $A$ if $A$ is finite, or to its Lebesgue measure if $A$ has positive measure.\\

{\bf Acknowledgements.} The author thanks Ciprian Demeter for reading this paper and giving several very useful suggestions. The authors also thanks Sean Prendiville for discussions on the applications of our main result. \\

\section{Brascamp-Lieb inequalities and ball-inflation lemmas}

Let $m$ be a positive integer. For $1\le j\le m$, let $V_j$ be a $d$-dimensional linear subspace of $\R^n$. Let also $\pi_j: \R^n\to V_j$ denote the orthogonal projection onto $V_j$. Define
\beq
\Lambda(f_1, f_2, ..., f_m)=\int_{\R^n} \prod_{j=1}^m f_j (\pi_j (x))dx,
\endeq
for $f_j: V_j\to \C$. We recall the following theorem due to Bennett, Carbery, Christ and Tao \cite{BCCT}.
\begin{thm}[\cite{BCCT}]\label{BL-inequality}
\label{bcct}
Given $p\ge 1$, the estimate
\beq\label{0510e2.2}
|\Lambda(f_1, f_2, ..., f_m)| \lesim \prod_{j=1}^m \|f_j\|_{p}
\endeq
holds if and only if $np=dm$ and the following Brascamp-Lieb transversality condition is satisfied
\beq\label{bl-transversality}
dim(V) \le \frac{1}{p} \sum_{j=1}^m dim(\pi_j(V)), \text{ for each linear subspace } V\subset \R^n.
\endeq
\end{thm}
An equivalent formulation of the estimate \eqref{0510e2.2} is
\beq\label{0510e2.4}
\|(\prod_{j=1}^m g_j\circ \pi_j)^{1/m} \|_q \lesim (\prod_{j=1}^m \|g_j\|_2)^{1/m},
\endeq
with $q=\frac{2n}{d}.$
The restriction that $p\ge 1$ becomes $dm\ge n$. In our proof, $m$ will always be a large constant, hence this condition is always satisfied. The transversality condition \eqref{bl-transversality} becomes
\beq\label{0510e2.5}
dim(V)\le \frac{n}{d m}\sum_{j=1}^m dim(\pi_j(V)), \text{ for each subspace } V\subset \R^n.
\endeq
Let us be more precise about the parameters in \eqref{0510e2.5}. We will take $n=5$ as our surface $\mc{S}$ lives in $\R^5$. The degree $m$ of multi-linearity will be chosen to be a large number. Our proof will make use of two different values of the parameter $d$: First of all, we will use $d=2$, which corresponds to that the surface $\mc{S}$ is two-dimensional; secondly, we also need to use $d=4$, as at certain stage of the proof, we will view $\mc{S}$ as a four-dimensional surface in $\R^5$. For instance, see Lemma \ref{main1} in Subsection \ref{subsection-torsion}.\\

Recall that the surface we are looking at is
$(t, s, \Phi_t(t, s), \Phi_s(t, s), \Phi(t, s)).$
Its tangent space is spanned by  
\beq\label{first-linear-space}
\begin{split}
n_1=(1, 0, \Phi_{tt}, \Phi_{st}, \Phi_{t}) \text{ and } n_2=(0, 1, \Phi_{ts}, \Phi_{ss}, \Phi_{s}).
\end{split}
\endeq
Moreover, we denote 
\beq\label{second-linear-space}
n_3=(0, 0, 1, 0, t) \text{ and } n_4=(0, 0, 0, 1, s).
\endeq
We will see from the following Lemma \ref{general-taylor} that these two vectors span the ``second order tangent space''. At a point $\xi\in [0, 1]^2$, let $V_{\xi}^{(1)}$ be the linear space spanned by $n_1(\xi)$ and $n_2(\xi)$ given in \eqref{first-linear-space}. Let $V_{\xi}^{(2)}$ be the linear space spanned by $n_1, n_2, n_3, n_4$ at the point $\xi$.\\

Let $K\in \N$ be a large number. It will be sent to infinity at the end of our proof. A $K$-square is defined to be a closed square of length $1/K$ inside $[0, 1]^2$. The collection of all dyadic $K$-squares will be denoted by $Col_K$.

\begin{prop}\label{linear-algebra}
Take $\Lambda \in \N$. Denote $m=\Lambda K$. Let $R_1, R_2, ..., R_m$ be different $K$-squares from $Col_{K}$.  For each $1\le i\le m$, choose one point $\xi_i\in R_i$. If we choose $\Lambda$ sufficiently large, independently on any parameter, then the transversality condition \eqref{0510e2.5} with $(d, n)=(2, 5)$ $($respectively $(4, 5))$ is satisfied for the collection of spaces $\{V_{\xi_j}^{(1)}\}_{j=1}^m$ $($respectively $\{V_{\xi_j}^{(2)}\}_{j=1}^m)$.
\end{prop}

We will prove Proposition \ref{linear-algebra} in the following two subsections. How to check the Brascamp-Lieb transversality condition \eqref{0510e2.5} seems to have become a big obstacle in obtaining new decoupling inequalities associated with surfaces of high co-dimensions. For instance see \cite{BDG-2} where a particular two dimensional surface in $\R^9$ is considered. The forthcoming argument that corresponds to the case $d=4$ in Subsection \ref{subsection3.2} further develops the idea introduced in \cite{BDG-2}. From our argument, in particular Lemma \ref{general-taylor}, it will become clear that more algebraic structures need to understood in order to push our current results to homogeneous polynomials of degrees higher than three.

\subsection{Proof of Proposition \ref{linear-algebra}: The case $d=2$}

In this subsection we prove the first part of Proposition \ref{linear-algebra}. Let $\pi^{(1)}_{\xi}(V)$ denote the projection of the space $V$ on $V_{\xi}^{(1)}$. We will show that 
\beq\label{0105e2.8}
dim(V)\le \frac{5}{2} dim(\pi^{(1)}_{\xi}(V)) \text{ almost surely in } \xi.
\endeq
Let us assume \eqref{0105e2.8} for a moment and see how it implies the transversality condition \eqref{0510e2.5}. First of all, if we define an exceptional set
\beq
E_V:=\{\xi\in [0, 1]^2: dim(V)> \frac{5}{2} dim(\pi^{(1)}_{\xi}(V))\},
\endeq 
then \eqref{0105e2.8} implies that $E_V$ lies inside the zero set of a polynomial of degree less than $10$. 
However, Wongkew's lemma \cite{Wo} says that the $\frac{10}{K}$-neighbourhood of the zero set of such a polynomial will intersect at most $C K$ squares in $Col_K$ for some large constant $C$. The desired  transversality condition \eqref{0510e2.5} follows immediately if we choose $\Lambda=100 C$.\\

{\bf Case $\dim(V)=1 \text{ or } 2$.} The desired estimate \eqref{0105e2.8} is reduced to 
\beq\label{0105e2.9}
dim(\pi^{(1)}_{\xi}(V))=1 \text{ almost surely.}
\endeq 
Suppose $V=\text{span}\{u\}$ with $u=(u_1, u_2, u_3, u_4, u_5)$. Then \eqref{0105e2.9} is equivalent to 
\beq
(u\cdot n_1, u\cdot n_2)\not\equiv (0, 0).
\endeq
We argue by contradiction. Suppose $(u\cdot n_1, u\cdot n_2)=(0, 0)$ for every $\xi\in [0, 1]^2$. By checking the constant terms in the polynomials $u\cdot n_1$ and $u\cdot n_2$, we obtain $u_1=u_2=0$. By checking the highest order terms, we obtain $u_5$=0. These two facts further imply that the cross product 
\beq
(\Phi_{tt}, \Phi_{st})\times (\Phi_{ts}, \Phi_{ss})
\endeq
is constantly zero. However, by a direct calculation, this contradicts to the assumption that the polynomial $\Phi$ is non-degenerate.\\

{\bf Case $\dim(V)=3 \text{ or } 4$.} We need to show that $dim(\pi^{(1)}_{\xi}(V))\ge 2$ almost surely. This is done via a direct calculation. Clearly the case $\dim(V)=3$ is more difficult. Suppose $V=\text{span}\{u, v, w\}$. Then the dimension of $\pi^{(1)}_{\xi}(V)$ is equal to the rank of the matrix 
\beq
\begin{pmatrix}
u\cdot n_1 &  v\cdot n_1 & w\cdot n_1\\
u\cdot n_2 &  v\cdot n_2 & w\cdot n_2
\end{pmatrix}
\endeq
We argue by contradiction and suppose that the determinants of all the two by two minors vanish constantly. We look at the two by two minor formed by the first two columns. The determinant of the matrix
\beq\label{0105e2.13}
\begin{pmatrix}
u_1+u_3 \Phi_{tt}+u_4 \Phi_{st}+u_5\Phi_t & v_1+v_3 \Phi_{tt}+v_4 \Phi_{st}+v_5\Phi_t\\
u_2+u_3 \Phi_{ts}+u_4 \Phi_{ss}+u_5\Phi_s & v_2+v_3 \Phi_{ts}+v_4 \Phi_{ss}+v_5\Phi_s
\end{pmatrix}
\endeq
vanishes constantly. Denote 
\beq
d_{i, j}:=\det
\begin{pmatrix}
u_i & u_j\\
v_i & v_j
\end{pmatrix}
\endeq
We first look at the third order term, that is 
\beq
\begin{split}
& d_{5, 4} \Phi_t \Phi_{ss}+d_{5, 3}\Phi_t \Phi_{ts}+d_{3, 5} \Phi_{tt}\Phi_s+d_{4 ,5} \Phi_{st}\Phi_s\\
& =\Big( d_{3, 5}\frac{\partial}{\partial_t}\big(\frac{\Phi_t}{\Phi_s}\big)+d_{4, 5}\frac{\partial}{\partial_s}\big(\frac{\Phi_t}{\Phi_s}\big)\Big)\Phi_s^2\equiv 0.
\end{split}
\endeq
This further implies $d_{3, 5}=d_{4, 5}=0$. Moreover, we know $d_{1, 2}=0$ by checking the constant term of the determinant of the matrix \eqref{0105e2.13}. This further implies that 
\beq\label{0105e2.16}
(u_5, v_5, w_5)=(0, 0, 0),
\endeq 
as otherwise we would derive a contradiction that $(u, v, w)$, when viewed as a matrix of order $3\times 5$, has rank two or smaller. \\

Substitute the identity \eqref{0105e2.16} into \eqref{0105e2.13}, and look at the second order term of the determinant of \eqref{0105e2.13}. We obtain 
\beq
d_{3, 4}\Phi_{tt}\Phi_{ss}-d_{3, 4}\Phi_{st}\Phi_{st}\equiv 0.
\endeq
By the non-degeneracy assumption on $\Phi$, we obtain that $d_{3, 4}=0.$ This, together with \eqref{0105e2.16} and $d_{1, 2}=0$, implies that the $3\times 5$ matrix $(u, v, w)$ has rank two or smaller. Contradiction.

\subsection{Proof of Proposition \ref{linear-algebra}: The case $d=4$}\label{subsection3.2}

We let $\pi^{(2)}_{\xi}(V)$ denote the projection of the space $V$ on $V_{\xi}^{(2)}$.
We need to show that 
\beq
dim(V)\le \frac{5}{4} dim(\pi^{(2)}_{\xi}(V)) \text{ almost surely}.
\endeq
This amounts to calculating the dimension of 
\beq\label{projection-2}
\{(u\cdot n_1, u\cdot n_2, u\cdot n_3, u\cdot n_4): u\in V\}.
\endeq
Following \cite{BDG-2}, we define linear spaces 
\beq
S_1=[t, s]; S_2=[\Phi_t(t, s), \Phi_s(t, s)] \text{ and } S_3=[\Phi(t, s)].
\endeq
We need the following version of Taylor's formula. 
\begin{lem}\label{general-taylor}
If $f\in S_3$, then 
\beq\label{general-taylor-2}
\Delta f(t, s)\approx f_t(t, s) \Delta t+f_s(t, s) \Delta s+t\cdot f_t(\Delta t, \Delta s)+s\cdot f_s(\Delta t, \Delta s).
\endeq
Here $\Delta f(t, s)=f(t+\Delta t, s+\Delta s)-f(t, s)$. The error produced by the approximate identity is a third order homogeneous polynomial in $\Delta t$ and $\Delta s$.
\end{lem}
\begin{proof}
By linearity, it suffices to consider $f(t, s)=\Phi(t, s)$. We calculate $\Phi(t+\Delta t, s+\Delta s)-\Phi(t, s)$ and view it as a homogeneous polynomial of four variables $t, s, \Delta t$ and $\Delta s$. First, we collection the linear terms with respect to $\Delta t$ and $\Delta s$. By the first order Taylor expansion, they are given by $\Phi_t(t, s) \Delta t+\Phi_s(t, s)\Delta s$, which are the former two terms on the right hand side of \eqref{general-taylor-2}. Next, we collect the quadratic terms with respect to $\Delta t$ and $\Delta s$. These terms must be linear in the variables $t$ and $s$. We apply the first order Taylor expansion again, with the roles of $(t, s)$ and $(\Delta t, \Delta s)$ exchanged, and obtain $t\Phi_t(\Delta t, \Delta s)+ s\Phi_s(\Delta t, \Delta s)$, which gives the latter two terms in \eqref{general-taylor-2}.
\end{proof}
This lemma can be written in the following equivalent way.
\beq
f(t, s)-f(t_0, s_0)\approx f_t(t_0, s_0) (t-t_0)+f_s(t_0, s_0) (s-s_0)+t_0\cdot f_t(t-t_0, s-s_0)+s_0\cdot f_s(t-t_0, s-s_0).
\endeq
According to this formula, let us consider 
\beq\label{0105e2.23}
f(t, s)=u_1 t+u_2 s+u_3 \Phi_t(t, s)+u_4 \Phi_s(t, s)+u_5 \Phi(t, s).
\endeq
At each point $\xi=(t_0, s_0)$, denote $\Delta t=t-t_0$ and $\Delta s=s-s_0$. We define 
\beq
(P_{\xi}f)(t, s)=f(\xi)+f_t(\xi)\cdot \Delta t+f_s(\xi)\cdot \Delta s + (u_3+u_5 t_0) \Phi_t(\Delta t, \Delta s)+(u_4+u_5 s_0) \Phi_s(\Delta t, \Delta s).
\endeq
Here we observe that 
\beq
P_{\xi} f=f \text{ for } f\in S_1\oplus S_2.
\endeq
We further define the canonical projection $\pi_{S_1\oplus S_2}$ onto the space $S_1\oplus S_2$. Hence 
\beq
\begin{split}
& (\pi_{S_1\oplus S_2} P_{\xi} f)(t, s)= (f_t(\xi)+(u_3+u_5t_0)\Phi_{tt}(-t_0, -s_0)+(u_4+u_5 s_0)\Phi_{st}(-t_0, -s_0)) t\\
& + (f_s(\xi)+(u_3+u_5t_0)\Phi_{ts}(-t_0, -s_0)+(u_4+u_5 s_0)\Phi_{ss}(-t_0, -s_0))s\\
& + (u_3+u_5 t_0) \Phi_t(t, s)+ (u_4+u_5 s_0) \Phi_s(t, s)\\
& = (u_1+u_5 \Phi_t(\xi)-u_5 t_0 \Phi_{tt}(\xi)-u_5 s_0 \Phi_{st}(\xi))t+(u_2+u_5 \Phi_s(\xi)-u_5 t_0 \Phi_{ts}(\xi)-u_5 s_0 \Phi_{ss}(\xi))s\\
&+ (u_3+u_5 t_0) \Phi_t(t, s)+ (u_4+u_5 s_0) \Phi_s(t, s).
\end{split}
\endeq
We can write 
\beq\label{projection-1}
\begin{split}
 \pi_{S_1\oplus S_2} P_{\xi} f =&(u_1+u_5 \Phi_t(\xi)-u_5 t_0 \Phi_{tt}(\xi)-u_5 s_0 \Phi_{st}(\xi),\\
& u_2+u_5 \Phi_s(\xi)-u_5 t_0 \Phi_{ts}(\xi)-u_5 s_0 \Phi_{ss}(\xi), u_3+u_5 t_0, u_4+u_5 s_0).
\end{split}
\endeq
Recall the choice of the function $f$ in \eqref{0105e2.23}.
Let us compare the vector in \eqref{projection-1} with the vector in \eqref{projection-2}, which is given by 
\beq
(u_1+u_3 \Phi_{tt}+u_4 \Phi_{st}+u_5 \Phi_t, u_2+u_3 \Phi_{ts}+u_4 \Phi_{ss}+u_5 \Phi_s, u_3+u_5 t_0, u_4+u_5 s_0).
\endeq
By some simple row and column transformations, we see that 
\beq
\begin{split}
& dim\big( \{(u\cdot n_1, u\cdot n_2, u\cdot n_3, u\cdot n_4): u\in V\}\big)\\
& =dim \big( \{\pi_{S_1\oplus S_2}P_{\xi}f: f\in V\} \big).
\end{split}
\endeq
Hence what we need to show becomes 
\beq
dim(V)\le \frac{5}{4} dim(\{\pi_{S_1\oplus S_2}P_{\xi}f: f\in V\}) \text{ almost surely}.
\endeq

{\bf Case $\dim(V)=1$.} The is the same as the case $\dim(V)=1$ and $d=2$.\\

{\bf Case $\dim(V)=2$.} We need to show that $dim(\pi_{S_1\oplus S_2}P_{\xi}(V))= 2$ almost surely. Argue by contradiction. Suppose $dim(\pi_{S_1\oplus S_2}P_{\xi}(V))\le 1$ everywhere. Then 
\beq
V=\pi_{S_1\oplus S_2}(V)\oplus S_3.
\endeq
Let us calculate the projection of $S_3$ on $S_1\oplus S_2$. Take 
\beq
f(t, s)=\Phi(t, s)=a t^3+bt^2 s+c t s^2+d s^3.
\endeq
Hence 
\beq
\pi_{S_1\oplus S_2}P_{\xi}f=(-3at_0^2-2bt_0s_0-cs_0^2, -bt_0^2-2ct_0s_0-3ds_0^2, t_0, s_0)
\endeq
As we know that 
\beq
V=\pi_{S_1\oplus S_2}(V)\oplus S_3,
\endeq
if we write $\pi_{S_1\oplus S_2}(V)=\text{span}\{u\}$ with $u=(u_1, u_2, u_3, u_4)$, then the dimension of $\pi_{S_1\oplus S_2}P_{\xi}(V)$ is equal  to the rank of the matrix 
\beq
\begin{pmatrix}
-3at_0^2-2bt_0s_0-cs_0^2 & -bt_0^2-2ct_0s_0-3ds_0^2 & t_0 & s_0\\
u_1 & u_2 & u_3 & u_4
\end{pmatrix}
\endeq
For every nonzero vector $u$, this matrix has rank two almost surely. \\

{\bf Case $\dim(V)=3$.} We need to show that $dim(\pi_{S_1\oplus S_2}P_{\xi}(V))= 3$ almost surely. Suppose not. Then by taking $\xi=(0, 0)$, we obtain that $dim(\pi_{S_1\oplus S_2}(V))=2$. Moreover, 
\beq
V=\pi_{S_1\oplus S_2}(V)\oplus S_3.
\endeq
Write $\pi_{S_1\oplus S_2}(V)=\text{span}\{u, v\}$ with $u=(u_1, u_2, u_3, u_4)$ and $v=(v_1, v_2, v_3, v_4)$. We need to show that the matrix 
\beq
\begin{pmatrix}
-3at_0^2-2bt_0s_0-cs_0^2 & -bt_0^2-2ct_0s_0-3ds_0^2 & t_0 & s_0\\
u_1 & u_2 & u_3 & u_4\\
v_1 & v_2 & v_3 & v_4
\end{pmatrix}
\endeq
has rank three almost surely. By calculating the determinants of all the $3\times 3$ minors, it is not difficult to see that this is indeed the case. \\

{\bf Case $\dim(V)=4$.} We need to show that $dim(\pi_{S_1\oplus S_2}P_{\xi}(V))= 4$ almost surely. Similar as above, we prove by contradiction. In the end, we need to show that the matrix 
\beq\label{matrix-4}
\begin{pmatrix}
-3at_0^2-2bt_0s_0-cs_0^2 & -bt_0^2-2ct_0s_0-3ds_0^2 & t_0 & s_0\\
u_1 & u_2 & u_3 & u_4\\
v_1 & v_2 & v_3 & v_4\\
w_1 & w_2 & w_3 & w_4
\end{pmatrix}
\endeq
has rank four almost surely. Argue by contradiction. Suppose the determinant of the above matrix vanishes constantly. By checking the linear terms in $t_0$ and $s_0$, we obtain that 
\beq
\det \begin{pmatrix}
u_1 & u_2 & u_3\\
v_1 & v_2 & v_3\\
w_1 & w_2 & w_3
\end{pmatrix}
=\det \begin{pmatrix}
u_1 & u_2 & u_4\\
v_1 & v_2 & v_4\\
w_1 & w_2 & w_4
\end{pmatrix}
=0.
\endeq
This implies that the two vectors $(u_1, v_1, w_1)$ and $(u_2, v_2, w_2)$ are linearly dependent. That the determinant of the matrix \eqref{matrix-4} vanishes constantly contradicts to the non-degeneracy of the polynomial $\Phi$. 

\subsection{Multi-linear Kakeya inequalities and ball-inflation lemmas}\label{subsection-torsion}

In Proposition \ref{linear-algebra}, we verified a transversality condition. As a consequence, we have the following multi-linear Kakeya inequality. 

\begin{lem}[Multi-linear Kakeya]\label{0720lemma3.5h}
Let $d_{1}=2$ and $d_{2}=4$. For every $\iota\in \{1, 2\}$, we have the following estimate: Let $M=\Lambda K$ where $\Lambda$ is the same as the one in Proposition \ref{linear-algebra}. Let $R_1, ..., R_M$ be different sets from $Col_K$. Consider $M$ families $\mc{P}_j$ consisting of rectangular boxes $P$ in $\R^5$, that we refer to as  plates,  having the following properties\\
1) For each $P\in \mc{P}_j$, there exits $\xi_j=(t_j, s_j)\in R_j$ such that $d_{\iota}$ sides of $P$ have lengths equal to $R^{1/2}$ and span $V^{(\iota)}_{\xi_j}$, while the remaining $(5-d_{\iota})$ sides have lengths $R$;\\
2) all plates are subsets of a ball $B_{4R}$ of radius $4R$.\\
Then we have the following inequality
\beq
\frac{1}{|B_{4R}|}\int_{B_{4R}} | \prod_{j=1}^M F_j |^{\frac{1}{M}\frac{5}{d_{\iota}}} \lesim_{\epsilon,\nu} R^{\epsilon}\left[ \prod_{j=1}^M  \left( \frac{1}{|B_{4R}|}|\int_{B_{4R}} F_j |\right)^{\frac{1}{M}} \right]^{\frac{5}{d_{\iota}}}
\endeq
for each function $F_j$ of the form
\beq
F_j=\sum_{P\in \mc{P}_j} c_P 1_P.
\endeq
The implicit constant does not depend on $R$ or $c_P$.
\end{lem}

Lemma \ref{0720lemma3.5h} is essentially due to Guth \cite{Guth}, and Bennett, Bez, Flock and Lee \cite{BBFL}. It follows from the Brascamp-Lieb inequalities in Theorem \ref{BL-inequality} via an induction-on-scales argument, after verifying the corresponding transversality conditions. Here we leave out the details. These multi-liner Kakeya inequalities have the following consequences.  
\begin{lem}[Ball-inflation lemmas]\label{main1}
Let $d_1=2$ and $d_2=4$. For every $\iota\in \{1, 2\}$, we have the following estimate: Let $R_1, ..., R_M$ be different squares from $Col_K$. Let $B$ be an arbitrary ball in $\R^5$ of radius $\rho^{-\iota-1}$. Let $\mc{B}$ be a finitely overlapping cover of $B$ with balls $\Delta$ of radius $\rho^{-\iota}$. For each $g:[0, 1]^2\to \C$, we have
\beq\label{0410e3.1}
\begin{split}
& \frac{1}{|\mc{B}|} \sum_{\Delta\in \mc{B}}\left[ \prod_{i=1}^M\left( \sum_{\substack{R'_i: \text{ square in } R_i\\ l(R'_i)=\rho}} \|E_{R'_i}g\|_{L^{\frac{d_{\iota}p}{5}}_{\#}(w_{\Delta})}^{\frac{d_{\iota}p}{5}} \right)^{\frac{5}{d_{\iota} p}} \right]^{\frac{p}{M}}\\
& \lesim_{\epsilon,\nu} \rho^{-\epsilon} \left[ \prod_{i=1}^M\left( \sum_{\substack{R'_i: \text{ square in } R_i\\ l(R'_i)=\rho}} \|E_{R'_i}g\|_{L^{\frac{d_{\iota}p}{5}}_{\#}(w_B)}^{\frac{d_{\iota}p}{5}} \right)^{\frac{5}{d_{\iota}p}} \right]^{\frac{p}{M}}.
\end{split}
\endeq
\end{lem}
\begin{proof}
We will follow the proof in Bourgain, Demeter and Guth \cite{BDG}, that is, we will derive Lemma \ref{main1} from Lemma \ref{0720lemma3.5h}. In order to apply Lemma \ref{0720lemma3.5h}, we need to check that the function $|E_{R'_i}g|$ is essentially a constant on a plate whose two short sides span the linear space $V^{(1)}_{\xi_j}$ for some $\xi_j\in R'_j$. Moreover, we need to check that for each ball $\Delta$ of radius $\rho^{-2}$, the function $\|E_{R'_i}g\|_{L_{\#}^{\frac{4p}{5}}(w_{\Delta})}$ is essentially a constant on a plate whose four short sides span the linear space $V^{(2)}_{\xi_j}$ for some $\xi_j\in R'_j$. The former statement follows from the standard Taylor expansion, and the latter one follows from Lemma \ref{general-taylor}. We comment here that this is what we meant previously by viewing the surface $\mc{S}$ as a four dimensional surface in $\R^5$. For the rest of the details, we refer to \cite{BDG}.
\end{proof}

The idea of ball-inflations originated from the work of Bourgain, Demeter and Guth \cite{BDG} (see Theorem 6.6 there). If we replace the $l^{\frac{d_{\iota}p}{5}}$ summations over $R'_i\subset R_i$ on both sides of \eqref{0410e3.1} by $l^2$ sums, essentially we arrive at the ball-inflation estimates that are proven in \cite{BDG}. Moreover, as has been pointed out above, the proof of Theorem 6.6 in \cite{BDG} also works for \eqref{0410e3.1}. However there are some subtle differences when it comes to applying these ball-inflation estimates in the iteration argument in Section \ref{section-iteration}.

In \cite{BDG}, the authors there used $l^2$ sums over $R'_i\subset R_i$, in order to prove certain sharp $l^2 L^p$ decoupling inequalities associated to moment curves. In our case, sharp $l^2L^p$ decouplings are no long able to imply good enough estimates as in Theorem \ref{main-1}. Indeed, this issue already appeared in earlier attempts of trying to push the argument of \cite{BDG} to higher dimensions, see \cite{BDG-2} and \cite{GZ18} (see \cite{BD2} for an even earlier work). In the present paper, we follow the way in which ball-inflation lemmas are formulated in the work \cite{GZ18} by Zhang and the author.

\section{Parabolic rescaling}
In this section we state the following result which is referred to as  parabolic rescaling.
\begin{lem}
\label{abc18}
Let $0<\delta<\sigma\le 1$. Then for each square $R\subset [0, 1]^2$ with side length $\sigma$ and each ball $B\subset \R^5$ with radius $\delta^{-3}$ we have
\beq
\|E_{R} g\|_{L^p(w_B)} \le B_{p,q}(\frac{\delta}{\sigma}) (\sum_{R'\subset R:\; l(R')=\delta} \|E_{R'} g\|_{L^{p}(w_B)}^{q})^{1/q}.
\endeq
\end{lem}
The proof of this lemma is standard, see for instance Proposition 7.1 from \cite{BD2}. One just needs to observe that our surface $\mc{S}$ is translation and dilation invariant, as can be seen via Lemma \ref{general-taylor}.

The parabolic rescaling lemma plays a determinant role in decoupling theory. It is used in every iteration step. First of all, it is used to run the Bourgain-Guth scheme, in order to show the equivalence between the linear and multilinear decoupling inequalities (Theorem \ref{abc37}). Secondly, it is used in the iteration scheme in Section \ref{section-iteration}  to conclude the desired decoupling inequality \eqref{desired-decoupling}.

\section{Linear versus multilinear decoupling}\label{2405sub3.2}

In this section we introduce a multi-linear version of the desired decoupling inequality. Recall that $K$ is a large number and $M=\Lambda K$. We denote by $B_{p, q}(\delta, K)$ the smallest constant such that
\beq\label{0720e3.20h}
\|(\prod_{i=1}^{M} E_{R_i} g)^{1/M}\|_{L^p(w_B)}\le B_{p, q}(\delta, K) \prod_{i=1}^{M} (\sum_{R_i'\subset R_i:\; l(R'_i)=\delta} \|E_{R'_i} g\|_{L^p(w_B)}^{q} )^{\frac{1}{qM}}.
\endeq
holds true for all distinct squares $R_i\in Col_K$, each ball  $B\subset \R^9$ of radius $\delta^{-3}$, and each $g:[0,1]^2\to\C$.

By H\"older's inequality, we see that the multi-linear decoupling constant $B_{p, q}(\delta, K)$ can be controlled by the linear decoupling constant $B_{p, q}(\delta)$. It turns out that, in the case $p=q$, the reverse direction also essentially holds true. That is,

\begin{thm}
\label{abc37}
For each $p\ge 2$ and $K\in \N$, there exists $\Omega_{K, p}>0$ and $\beta(K, p)>0$ with
\beq
\lim_{K\to \infty} \beta(K, p)=0, \text{ for each }p,
\endeq
such that for each small enough $\delta$, we have
\beq
B_{p, p}(\delta)\le \delta^{-\beta(K, p)-2(\frac{1}{2}-\frac{1}{p})}+ \Omega_{K, p} \log_K \big(\frac{1}{\delta} \big)\max_{\delta \le \delta'\le 1}(\frac{\delta'}{\delta})^{2(\frac{1}{2}-\frac{1}{p})} B_{p, p}(\delta', K).
\endeq
\end{thm}

The proof of this theorem is standard, and is essentially the same as that of Theorem 8.1 from \cite{BD2}. Hence we leave it out.

\section{Iteration}\label{section-iteration}
In this section, we run the final iteration argument. The consequence of this iteration, combined with Theorem \ref{abc37}, will lead to the desired decoupling inequality \eqref{desired-decoupling}. \\

There will be two terms that are involved in the iteration procedure. They are 
\beq\label{iteration-term-2}
D_p(q, B^r):=(\prod_{i=1}^M \sum_{J_{i, q\subset R_i}}\|E_{J_{i, q}}g\|_{L^p_{\#}(w_{B^r})}^p)^{\frac{1}{pM}}
\endeq
and
\beq\label{iteration-term-1}
A_{p} (q, B^r, s)=\Big( \frac{1}{|\mc{B}_s(B^r)|} \sum_{B^s\in \mc{B}_s(B^r)} D_{2}(q, B^s)^{p} \Big)^{1/p}.
\endeq
Here for a positive number $r$, we use $B^r$ to denote a ball of radius $\delta^{-r}$, and $\mc{B}_s(B^r)$ denotes a finitely overlapping collection of balls $B^s$ that lie inside of a ball $B^r$.
 In the notation $J_{i, q}$, the index $i$ indicates that this square lies in $R_i$, and $q$ indicates that the square $J_{i, q}$ has side length $\delta^{q}$.

Define $\alpha_1, \alpha_2,\beta_2\in (0,1)$ as follows
$$
\frac{1}{\frac{2p}{5}}=\frac{\alpha_1}{\frac{4p}{5}}+\frac{1-\alpha_1}{2},
$$

$$
\frac{1}{\frac{4p}{5}}=\frac{\alpha_2}{p}+\frac{1-\alpha_2}{6},
$$

$$
\frac{1}{6}=\frac{1-\beta_2}{2}+\frac{\beta_2}{\frac{4p}{5}}.
$$
We will start our iteration with the term 
\beq
A_{p}(1, B^3, 1)=\Big( \frac{1}{|\mc{B}_1(B^3)|} \sum_{B^1\in \mc{B}_1(B^3)} D_{2}(1, B^1)^{p} \Big)^{1/p}.
\endeq
By H\"older's inequality, it can be bounded by 
\beq\label{180419e5.4}
\delta^{-2(\frac{1}{2}-\frac{5}{2p})}\Big( \frac{1}{|\mc{B}_1(B^3)|} \sum_{B^1\in \mc{B}_1(B^3)} D_{\frac{2p}{5}}(1, B^1)^{p} \Big)^{1/p}.
\endeq
We apply Lemma \ref{main1} with $\iota=1$ to \eqref{180419e5.4} and bound it by
\beq\label{180419e5.5}
\delta^{-2(\frac{1}{2}-\frac{5}{2p})-\epsilon}\Big( \frac{1}{|\mc{B}_2(B^3)|} \sum_{B^2\in \mc{B}_2(B^3)} D_{\frac{2p}{5}}(1, B^2)^{p} \Big)^{1/p}.
\endeq
By H\"older's inequality, the right hand side of \eqref{180419e5.5} can be dominated by
\beq\label{0410e4.4}
\delta^{-2(\frac{1}{2}-\frac{5}{2p})-\epsilon}\Big( \frac{1}{|\mc{B}_2(B^3)|} \sum_{B^2\in \mc{B}_2(B^3)} D_{\frac{4p}{5}}(1, B^2)^{p} \Big)^{\frac{\alpha_1}{p}} \Big( \frac{1}{|\mc{B}_2(B^3)|} \sum_{B^2\in \mc{B}_2(B^3)} D_{2}(1, B^2)^{p} \Big)^{\frac{1-\alpha_1}{p}}.
\endeq
By $L^2$ orthogonality, this can be bounded by 
\beq\label{0107e5.6}
\delta^{-2(\frac{1}{2}-\frac{5}{2p})-\epsilon}\Big( \frac{1}{|\mc{B}_2(B^3)|} \sum_{B^2\in \mc{B}_2(B^3)} D_{\frac{4p}{5}}(1, B^2)^{p} \Big)^{\frac{\alpha_1}{p}} A_p(2, B^3, 2)^{1-\alpha_1}.
\endeq
In the next step, we apply Lemma \ref{main1} with $\iota=2$ and obtain 
\beq\label{0107e5.8}
\delta^{-2(\frac{1}{2}-\frac{5}{2p})-\epsilon}D_{\frac{4p}{5}}(1, B^3)^{\alpha_1} A_p(2, B^3, 2)^{1-\alpha_1}.
\endeq
The last term $A_p(2, B^3, 2)^{1-\alpha_1}$ is ready for iteration. We further process the $D$-term. By H\"older's inequality  
\beq
D_{\frac{4p}{5}}(1, B^3) \lesim D_6(1, B^3)^{1-\alpha_2} D_p(1, B^3)^{\alpha_2}.
\endeq
The second term on the right hand side is already of the form of the term in the decoupling inequality \eqref{decoupling-result}, and it will not be further processed. It is the former term on the right hand side that needs further process.

Notice that in the term $D_6(1, B^3)$, we are dealing with terms $\|E_{J_{i, 1}}g\|_{L^6_{\#}(w_{B^3})}$. By the uncertainty principle, such a ball of radius $\delta^{-3}$ is not able to distinguish the surface $\mc{S}$ from 
\beq\label{2in4surface}
\{(t, s, \Phi_t(t, s), \Phi_s(t, s), 0): (t, s)\in J_{i, 1}\}
\endeq
under certain affine transformations. By the $l^6 L^6$ decoupling estimate for the surface \eqref{2in4surface} obtained in \cite{BD-2in4},  we obtain 
\beq
D_6(1, B^3) \lesim \delta^{-(\frac{1}{2}-\frac{1}{6})-\epsilon}D_6(\frac{3}{2}, B^3).
\endeq
By H\"older's inequality, this can be further bounded by 
\beq
\begin{split}
D_6(1, B^3)  & \lesim \delta^{-(\frac{1}{2}-\frac{1}{6})-\epsilon}D_6(\frac{3}{2}, B^3)\\
		& \lesim  \delta^{-(\frac{1}{2}-\frac{1}{6})-\epsilon}D_2(\frac{3}{2}, B^3)^{1-\beta_2} D_{\frac{4p}{5}}(\frac{3}{2}, B^3)^{\beta_2}\\
		& \lesim \delta^{-(\frac{1}{2}-\frac{1}{6})-\epsilon} D_2(3, B^3)^{1-\beta_2} D_{\frac{4p}{5}}(\frac{3}{2}, B^3)^{\beta_2}.
\end{split}
\endeq
In the last step, we applied $L^2$ orthogonality. In the end, what we have obtained so far can be organised as 
\beq\label{0410e4.11}
\begin{split}
& A_p(1, B^3, 1)  \lesim_\epsilon \delta^{-\epsilon-2(\frac{1}{2}-\frac{5}{2p}) -(\frac{1}{2}-\frac{1}{6})\alpha_1(1-\alpha_2)}\times \\
		& A_p(2, B^3, 2)^{1-\alpha_1} A_p(3, B^3, 3)^{\alpha_1 (1-\alpha_2)(1-\beta_2)}D_{\frac{4p}{5}}(\frac{3}{2}, B^3)^{\alpha_1 (1-\alpha_2)\beta_2} D_p(1, B^3)^{\alpha_1 \alpha_2}.
\end{split}
\endeq
Now we run this iteration procedure for $r$ many times. For all balls $B$ of radius $\delta^{-2\cdot (\frac{3}{2})^r}$, we have 
\beq\label{0410e4.12}
\begin{split}
& A_p(1, B, 1) \lesim_{\epsilon,r} (\frac{1}{\delta})^{\epsilon+2(\frac{1}{2}-\frac{5}{2p})} \times\underbrace{\prod_{i=0}^{r-1} (\frac{1}{\delta})^{(\frac{3}{2})^i (\frac{1}{2}-\frac{1}{6}) \alpha_1 (1-\alpha_2) [(1-\alpha_2)\beta_2]^{i}}}_{l^{6} L^6 \text{ decoupling }} \\
& \times A_p(2, B, 2)^{1-\alpha_1}  D_{\frac{4p}{5}}\Big((\frac{3}{2})^r, B \Big)^{\alpha_1 [(1-\alpha_2)\beta_2]^r}\\
& \left(\prod_{i=1}^r A_p(2 (\frac{3}{2})^i, B, 2 (\frac{3}{2})^i)^{\alpha_1 (1-\alpha_2)(1-\beta_2)[(1-\alpha_2)\beta_2]^{i-1}}\right)\left(\prod_{i=0}^{r-1}D_p((\frac{3}{2})^i, B)^{\alpha_1 \alpha_2 [(1-\alpha_2)\beta_2]^i} \right).
\end{split}
\endeq
Define
\beq\label{0410e4.13}
\begin{split}
& \gamma_0=1-\alpha_1; \gamma_i=\alpha_1 (1-\alpha_2)(1-\beta_2)[(1-\alpha_2)\beta_2]^{i-1}, \text{ for } 1\le i\le r;\\
& b_i=2\cdot (\frac{3}{2})^i, \text{ for } 0\le i\le r;\\
& \tau_r=\alpha_1 [(1-\alpha_2)\beta_2]^r; \tau_i=\alpha_1 \alpha_2 [(1-\alpha_2)\beta_2]^i, \text{ for }0\le i\le r-1;\\
& w_i= \frac{1-\alpha_2}{2\alpha_2}\tau_i, \text{ for } 0\le i\le r-1.
\end{split}
\endeq
We can write using H\"older
$$D_{\frac{4p}{5}}\Big((\frac{3}{2})^r, B \Big)\lesssim D_{p}\Big((\frac{3}{2})^r, B \Big).$$
With these,  the estimate \eqref{0410e4.12} becomes
\beq\label{}
\begin{split}
A_p(1, B, 1)& \lesim_{r,\epsilon} (\frac{1}{\delta})^{\epsilon+2(\frac{1}{2}-\frac{5}{2p})}\Big( \prod_{i=0}^{r-1} (\frac{1}{\delta})^{ (\frac{1}{2}-\frac{1}{6}) b_i w_i} \Big)\times  \\ &\Big(\prod_{i=0}^r A_p(b_i, B, b_i)^{\gamma_i} \Big)\Big( \prod_{i=0}^{r}D_p(\frac{b_i}{2}, B)^{\tau_i} \Big)
\end{split}
\endeq
Using this and a simple rescaling argument, we can rewrite \eqref{0410e4.12} as follows
\beq\label{0107e5.17}
\begin{split}
A_p(u, B, u)& \lesim_{r,\epsilon} (\frac{1}{\delta})^{\epsilon+2u(\frac{1}{2}-\frac{5}{2p})}\Big( \prod_{i=0}^{r-1} (\frac{1}{\delta})^{ (\frac{1}{2}-\frac{1}{6}) u b_i w_i} \Big)\times  \\ &\Big(\prod_{i=0}^r A_p(b_i u, B, b_i u)^{\gamma_i} \Big)\Big( \prod_{i=0}^{r}D_p(\frac{b_i u}{2}, B)^{\tau_i} \Big).
\end{split}
\endeq
Here $B$ stands for a ball of radius $\delta^{-3}$, and $u$ is a sufficiently small positive constant such that $u\cdot (\frac{3}{2})^r\le 1$. \\

In the end, we iterate \eqref{0107e5.17}. To iterate, we will dominate each  $A_p(ub_i, B, u b_i)$ again by using \eqref{0107e5.17}. To enable such an iteration, we need to choose $u$ to be even smaller. Let $M$ be a large integer. Choose $u$ such that
\beq
[2(\frac{3}{2})^r]^M u\le 2.
\endeq
 This allows us to iterate \eqref{0107e5.17} $M$ times. To simplify the iteration, we bound all the powers of $\frac{1}{\delta}$ by 
 \beq
2u(\frac{1}{2}-\frac{5}{2p})+\left( \sum_{i=0}^{\infty} u (\frac{1}{2}-\frac{1}{6}) b_i w_i \right).
\endeq
By a direct calculation,
\beq\label{0107e5.20}
\sum_{j=0}^{\infty} b_j w_j=\frac{3 (-5 + p)}{2 (15 - 10 p + p^2)}.
\endeq
Moreover,
\beq\label{0107e5.22}
\sum_{j=0}^{\infty} b_j \tau_j=\frac{75 - 25 p + 2 p^2}{15 - 10 p + p^2}.
\endeq
If we define
\beq
\lambda_0:= 2(\frac{1}{2}-\frac{5}{2p})+(\frac{1}{2}-\frac{1}{6}) \frac{3 (-5 + p)}{2 (15 - 10 p + p^2)},
\endeq
then \eqref{0107e5.17} can be rewritten as follows
\beq\label{0405e2.45}
\begin{split}
A_p(u, B, u) \lesim_{r,\epsilon} \delta^{-\epsilon-u \lambda_0}  \Big(\prod_{i=0}^r A_p(ub_i, B, ub_i)^{\gamma_i} \Big)\Big( \prod_{i=0}^{r} D_p(\frac{u b_i}{2}, B)^{\tau_i} \Big),
\end{split}
\endeq
for every ball $B$ of radius $\delta^{-3}$. Now we have arrived precisely at the estimate (6.51) from \cite{BDG-2}. The calculation there, from page 27 to page 30, can be repeated line by line. In the end, we obtain that 
\beq
\log_{\frac{1}{\delta}}B_{9, 9}(\delta)\le \lim_{p\to 9} \frac{\lambda_0(p)}{\frac{1}{2}(\sum_{j=0}^{\infty} b_j \tau_j(p))}.
\endeq
By plugging in the calculation \eqref{0107e5.20}--\eqref{0107e5.22}, we will be able to conclude the desired decoupling inequality \eqref{desired-decoupling}.

\hspace{1cm}

\noindent Department of Mathematics, Indiana University, 831 East 3rd St., Bloomington IN 47405\\
\emph{Email address}: shaoguo@iu.edu

\end{document}